\documentclass[11pt]{amsart}
\usepackage{amscd, amsmath}
\usepackage{amssymb}
\usepackage[margin=3.5cm]{geometry}

\newtheorem{theorem}{Theorem}[section]
\newtheorem{lemma}[theorem]{Lemma}
\theoremstyle{definition}

\theoremstyle{remark}
\newtheorem{remark}[theorem]{Remark}
\numberwithin{equation}{subsection}
\theoremstyle{plain}

\newtheorem{cor}{Corollary}

\def\F{\mathbb F}

\def\V{\mathrm V}

\def\t{\mathfrak t}

\def\E{\mathbb E}
\def\F{\mathrm F}

\def\E{\mathrm E}
\def\char{\rm{char}}

\def\U{\rm U}

\newcommand{\secref}[1]{Section~\ref{#1}}
\newcommand{\thmref}[1]{Theorem~\ref{#1}}
\newcommand{\lemref}[1]{Lemma~\ref{#1}}

\newcommand{\eqnref}[1]{~{\textrm(\ref{#1})}}

\begin{document}
\title[Conjugate real classes in general linear group]{Conjugate Real Classes in General Linear Groups} 
\author{Krishnendu Gongopadhyay}
\address{Indian Institute of Science Education and Research (IISER) Mohali, Sector 81, SAS Nagar, Punjab 140306, India}
\email{krishnendu@iisermohali.ac.in, krishnendug@gmail.com}
\author{Sudip Mazumder}
\address{Department of Mathematics, Jadavpur University, Jadavpur, Kolkata 700032 }
\email{topologysudip@gmail.com}
\author{Sujit Kumar Sardar}
\address{Department of Mathematics, Jadavpur University, Jadavpur, Kolkata 700032 }
\email{sksardarjumath@gmail.com}
\date{\today}
\subjclass[2010]{Primary 20E45; Secondary  20G15, 15A04}
\keywords{real elements, general linear group}
\begin{abstract}

Let $\F$ be a field with a non-trivial involution $c: \alpha \mapsto \alpha^c$. An element $g \in {\rm GL}_n(\F)$ is called $c$-real if it is conjugate to $(g^c)^{-1}$.  We prove that for $n \geq 2$, $g \in {\rm GL}_n(\F)$ is $c$-real if and only if it has a representation in some unitary group of degree $n$ over $\F$. 
\end{abstract}
\maketitle

\section{Introduction}
Let $\F$ be a (commutative) field.  
Let $c: \alpha \mapsto \alpha^c$ be a fixed \emph{involution}  on $\F$, i.e.  $c$ is an automorphism of $\F$ satisfying $(\alpha^{c})^c=\alpha$ for all $\alpha$ in $\F$. The involution is called \emph{non-trivial} if it is different from the identity automorphism.  Let $\V$ be a finite dimensional vector space over $\F$ of dimension at least $2$.

Recall that a $c$-sesquilinear form on $\V$ is a bi-additive map $\sigma: \V \times \V \to \F$ such that for all $\alpha$, $\beta \in \F$, and $v, w \in \V$, $\sigma(\alpha v, \beta w)=\alpha \beta^c \sigma(v, w)$. A  $c$-sesquilinear form is called $c$-\emph{hermitian} if for all $v, w \in V$, 
$\sigma( v, w)=\sigma( w, v)^c$.   The isometry group consisting of linear transformations that preserve a $c$-hermitian form is called a \emph{unitary group}.

\medskip An element $g$ in an abstract group $G$ is called \emph{real} or \emph{reversible} if $g$ is conjugate to $g^{-1}$ in $G$. Real elements appear naturally in representation theory, geometry and dynamics. There have been an ongoing activity to understand real elements from several point of views, for example see the recent articles,  \cite{asv}, \cite{dolfi}, \cite{gs}, \cite{g}, \cite{gp}, \cite{hl},  \cite{jnk}, \cite{kk},  \cite{ks}, \cite{or},  \cite{st}, \cite{tz}, \cite{r}, \cite{vg}. For an up to date exposition of real elements from different point of views, we refer to the recent monograph  \cite{osho}. 

\medskip The notion of real elements has a natural extension to linear groups over fields with involutions as follows. Let $\V$ has dimension $n$ over $\F$.  Let ${\rm GL}_n(\F)$ be the general linear group, i.e. group of all invertible linear maps on $\V$. Let $G$ be a $c$-invariant linear subgroup of ${\rm GL}_n(\F)$. For an element $g=(a_{ij})$ in $G$, let $g^c=(a^c_{ij})$. An element $g$ in $G$ is called \emph{conjugate  real} or, simply \emph{$c$-real}  if $g$ is conjugate to $(g^c)^{-1}$ in $G$. When $c$ is identity, it matches with the notion of reality in groups. It would be curious to investigate $c$-reality in several classes of linear groups over a field with involution and,  to compare it with the notion of reality. Some notion of twisted reality, mostly transpose-reality,  is implicit in some recent works related to group representations, for example see \cite{vinroot1}, \cite{rs1}.  Investigation of the notion of $c$-reality seems missing in the literature. We consider $c$-reality in the general linear group ${\rm GL}_n(\F)$ in this paper.

\medskip The $c$-reality of a linear map has close connection with the \emph{self-duality} of its invariant factors.  Let $\overline \F$ be the algebraic closure of $\F$.
Let $f(x)=\sum_{i=0}^d a_i x^i$, $a_d=1$, be a monic polynomial of degree $d$ over $\F$ such that $0$, $1$ or $-1$ are not its roots. 
The \emph{dual} of $f(x)$ is defined to be the polynomial $f^{\ast}(x)=(f(0)^c)^{-1}x^d f^c(x^{-1})$, where $f^c(x)=\sum_{i=0}^d a_i^c x^i$.  Thus,  $f^{\ast}(x)=\frac{1}{a_0^c} \sum_{i=0}^d a_{d-i}^c x^{i}$. In other  words, if $\alpha$ in $\overline \F$ is a root of $f(x)$ with multiplicity $k$, then $(\alpha^c)^{-1}$ is a root of $f^{\ast}(x)$ with the same multiplicity. The polynomial $f(x)$ is said to be \emph{self-dual} if  $f(x)=f^{\ast}(x)$. For notational convenience, we shall slightly extend this notion and will call a polynomial $g(x)$ \emph{self-dual} if $g(x)$ is either a power of $x \pm 1$, or self-dual in the above sense. 

\medskip Our main theorem is the following. 

\begin{theorem}\label{mainthm} Let $\F$ be a field with a non-trivial involution $c$. Let  $g$ be an element in ${\rm GL}_n(\F)$, $n \geq 2$. Then $g$ is $c$-real if and only if it has a unitary representation of dimension $n$. 
\end{theorem}
We prove this theorem in \secref{mpr}. In \secref{if2}, we have proved a result related to the existence of invariant form under a linear map. This result is crucial for the proof of the main theorem,  and is of independent interest as well. 

\section{Existence of Invariant Forms}\label{if2}
In this section we investigate the following problem that has seen some attention in the literature: {\it given an invertible linear map $T: \V \to \V$, when does the vector space $\V$ over $\F$ admit a $T$-invariant non-degenerate $c$-hermitian  form}? The solution to this problem will be crucial in the proof of \thmref{mainthm}.

\medskip Sergeichuk \cite{s1, s2, s3}  solved the above problem assuming that the characteristic of $\F$ is  different from two. Using a different approach, Gongopadhyay and Kulkarni \cite{gk} obtained conditions for an invertible linear map to admit an invariant non-degenerate quadratic and symplectic form assuming that the underlying field is of large characteristic and, this was extended in \cite{gm} over fields of characteristic different from two to derive conditions for an invertible linear map to admit an invariant non-degenerate $c$-hermitian form. All these works assumed that the characteristic of the underlying field is different from two.  When $c$ is trivial, the characteristic two case has been addressed by de Seguins Pazzis \cite{dsp}.   In this section we work out the remaining case of non-trivial $c$ in characteristic two. 

  \medskip Recall that a $T$-invariant subspace of $\V$ is said to be \emph{indecomposable} with respect to $T$, or simply \emph{$T$-indecomposable} if it can not be expressed as a direct sum of  proper $T$-invariant subspaces.  If $\V$ is $T$-indecomposable, $T$ is called \emph{cyclic}. It is well-known from the structure theory of linear operators that $\V$ can be written as a direct sum $\V=\bigoplus_{i=1}^m \V_i$, where each $\V_i$ is $T$-indecomposable for $i=1,2, \ldots, m$, and each pair $(\V_i, T|_{\V_i})$ is  isomorphic to $(\F[x]/(p(x)^k), \mu_x)$, where $p(x)$ is an irreducible monic factor of the minimal polynomial of $T$, and $\mu_x$ is the operator 
$[u(x)] \mapsto [xu(x)]$, for eg. see \cite{jac, kul}.  Such $p(x)^k$ is an \emph{elementary divisor} of $T$. If $p(x)^k$ occurs $d$ times in the decomposition, we call $d$ the \emph{multiplicity} of the elementary divisor $p(x)^k$.

\medskip 
With the notions as above, we prove the following result in this section that will be used in the proof of the main theorem. 

\begin{theorem} \label{mainthm1}
Let $\F$ be a field with a non-trivial involution $c$.  
 Then $\V$ admits a $T$-invariant non-degenerate $c$-hermitian form if and only if  an elementary divisor of $T$ is either self-dual, or its dual is also an elementary divisor with the same multiplicity.
\end{theorem}

When characteristic of $\F$ is different from two, the above theorem was proved in \cite{gm}. In this section,  we shall extend the proof to arbitrary characteristic. We shall prove the following result that combining with \cite[Theorem 1.1]{gm},  gives the above theorem. 
\begin{theorem}\label{mainth}
Let $\F$ be a field of characteristic 2,  and let $c: \alpha \mapsto \alpha^c$ be an involution on $\F$. 
 Then $\V$ admits a $T$-invariant non-degenerate $c$-hermitian form if and only if the following conditions hold. 

\begin{enumerate}
\item[(i)] An  elementary divisor of $T$ is either self-dual, or its dual is also an elementary divisor with the same multiplicity.

\item[(ii)]  If $c$ is identity and $(x-1)^m$ is an elementary divisor, then either $m$ is even, or if $m$ is odd, then multiplicity of the elementary divisor must be an even number. 
\end{enumerate}
\end{theorem}

To prove \thmref{mainth}, we first prove \lemref{lem1}. When characteristic of $\F$ is different from two, a version of the lemma was proved by Sergeichuk \cite{s2, s3}, and an easy proof was sketched in \cite{gm}.  Our proof below is motivated by a computational idea following \cite{kt}, where it has been used to investigate dimensions of the spaces of bilinear forms.  

\begin{lemma}\label{lem1}
 Let $\V$ be a finite dimensional vector space over a field   $\F$. 
Let $T: \V \to \V$ be a non-unipotent, cyclic, self-dual linear transformation. Then there exists a $T$-invariant non-degenerate $c$-hermitian form on $\V$.
\end{lemma}

\begin{proof}  Suppose $\dim \V=k$. Since $T$ is cyclic, $\chi_{T}(x)=m_{T}(x)$. Let $m_{T}(x)=p(x)^d$ be the minimal polynomial of $T$, where $p(x)$ is irreducible over $\F$, $p(x) \neq x \pm 1$.  Let $m_T(x)=p(x)^d=\sum_{i=0}^{ k }d_ix^{i}$.  As $T$ is self-dual then $d_0d_0^c=1$, $d_i d_0^c=d^{c}_{ k -i}, ~ 1\leq i\leq  k -1$.  Further since $T$ is cyclic, there is a vector $v$ in $\V$ such that the $T$-orbit of $v$ spans $\V$, i. e. the set  
$$B=\{e_1 =v,~ e_2=Tv, \ldots,~ e_{ k } =T ^{ k -1}v\}$$
forms a basis of $\V$ and $T^{ k }v=-\sum_{i=0}^{k-1} d_{i}T^{i}v$.

 The linear transformation  of $T$ with respect to $\{e_1,e_2,\ldots, e_k\}$ is given by the matrix:  
 $$T=\begin{pmatrix}
0 & 0 & 0 & \ldots & 0 & -d_0 \\1& 0 & 0 & \ldots & 0  & -d_1\\ 0 &  1& 0& \ldots & 0 & - d_2 \\ \vdots &  & \ddots& \ddots &  & \vdots\\ 0 & 0 & 0 & \ldots & 0 & - d_{ k -2}\\0 & 0 & 0 & \ldots & 1 & - d_{ k -1} \end{pmatrix}.$$
Note that when characteristic of $\F$ is $2$, $d_i=-d_i$ above for all $i$. 

Suppose $H$ is a sesquilinear $T$-invariant form. Then, with respect to the basis $B$,  it is given by the matrix 
 $X=(H(e_{i},e_{j}))$. The 
$T$-invariance of $H$ is equivalent to  the relation $(T^c)^t H T=H$. 
Thus, 
$$(H(e_{i},e_{j}))=X=(T^{c})^{t}X T =H(Te_{i},Te_{j}).$$ 
Let $H(e_i, e_j)=h_{ij}$, then it follows that $h_{ij}=h_{(i+1)(j+1)}$ for all  $ i,j,~1\leq i,j\leq  k -1$. In view of this,  for simplicity of notation, let $H(e_{1},e_{i})=x_{i}, ~  1\leq i\leq  k $ and $H(e_{j},e_{1})=y_{j},~ 1\leq j \leq k $. Clearly, $x_{1}=y_{1}$. So,
\begin{equation} \label{hma}  X=\begin{pmatrix} x_{1} & x_{2} & x_{3} & \ldots & x_{ k }\\ y_{2} & x_{1} & x_{2} & \ldots & x_{ k -1} \\ \vdots & \vdots & \vdots& \vdots & \vdots \\ y_{ k } & y_{ k -1} & y_{ k -2} & \ldots &x_{1} \end{pmatrix}.\end{equation} 
Now for  $1\leq i\leq [\frac{k}{2}]$, we get \begin{eqnarray*}x_{ k -i+1}&=& H(e_{i+1},T^{ k }v)\\&= &H(e_{i+1},-\sum_{j=0}^{ k -1}d_{j}e_{j+1})
=-\sum _{j=0}^{ k -1}H(e_{i+1},e_{j+1})d_{j}^{c}\\
&=&-\sum_{j=0}^{i-1}H(e_{i+1},e_{j+1})d_{j}^{c}-\sum_{j=i}^{ k -1}H(e_{i+1},e_{j+1})d_{j}^{c}\\
&=&-\sum_{j=i}^{ k -1}H(e_{1},e_{j-i+1})d_{j}^{c}-\sum_{j=0}^{i-1}H(e_{i+1},e_{j+1})d_{j}^{c}\\
&=&-\sum_{j=i}^{ k -1}x_{j-i+1}d_{j}^{c}-\sum_{j=0}^{i-1}H(e_{i-j+1},e_{1})d_{j}^{c}.
\end{eqnarray*}

 We get similar expression for $y_{k-i+1}$, and we have,  for $1\leq i \leq [\frac{k}{2}]$,
 $$x_{ k -i+1}=-\bigg(\sum_{j=i}^{ k -1}x_{j-i+1}d_{j}^{c}+\sum_{j=0}^{i-1}y_{i-j+1}d_{j}^{c}\bigg),$$
$$y_{ k -i+1}=-\bigg(\sum _{j=i}^{ k -1}y_{j-i+1}d_{j}+\sum_{j=0}^{i-1}x_{i-j+1}d_{j}\bigg).$$

\medskip {\it Suppose $k$ is even: $k=2n$}. For $i=t \leq n $, using $d_j d_0^c=d^{c}_{ k -j}, ~ 1\leq j \leq  k -1$, the  above equations can be written as:  
$$x_{ 2n -t+1}=-\bigg(y_{t+1}d_{0}^{c}+\sum_{j=1}^{t-1}(y_{t-j+1}d_j^c+x_{ 2n-j-t+1}d_{j}d_0^c)+\sum_{j=t}^{n-1}(x_{j-t+1} d_j^c+x_{ 2n-j-t+1}d_{j}d_0^c)+x_{n-t+1} d_{n}^{c}\bigg).$$

$$y_{ 2n -t+1}=-\bigg(x_{t+1}d_{0}+\sum_{j=1}^{t-1}(x_{t-j+1}d_j+y_{ 2n -j-t+1}d_{j}^cd_0)+\sum_{j=t}^{n-1}(y_{j-t+1}d_j+y_{ 2n-j-t+1} d_{j}^cd_0)+y_{n-t+1} d_{n} \bigg).$$
Thus,  $x_{n+1}, \cdots x_{ 2n },~y_{n+1},\ldots,~y_{ 2n }$ are expressible in terms of $x_{1},\ldots,x_{n},~y_{2},\ldots,y_{n}$. 

\medskip {\it Suppose $k$ is odd: $k=2n-1$}. For $i=t \leq n-1$, using $d_j d_0^c=d^{c}_{ k -j}, ~ 1\leq j \leq  k -1$, the  above equations can be written as:
$$x_{ 2n -t}=-\bigg(y_{t+1}d_{0}^{c}+\sum_{j=1}^{t-1}(y_{t-j+1} d_j^c+x_{ 2n -j-t}d_{j} d_0^{c})+\sum_{j=t}^{n-1}(x_{j-t+1}d_j^c+x_{ 2n -j-t}d_{j}d_0^{c})\bigg);$$

$$y_{ 2n -t}=-\bigg(x_{t+1}d_{0}+\sum_{j=1}^{t-1}(x_{t-j+1}d_j+y_{ 2n -j-t}d_{j}^c d_0)+\sum_{j=t}^{n-1}(y_{j-t+1}d_j+y_{ 2n -j-t}d_{j}^c d_0)\bigg). $$
Thus,  $x_{n+1}, \cdots x_{ 2n-1},~y_{n+1},\ldots,~y_{ 2n-1 }$ are expressible in terms of $x_{1},\ldots$, $x_{n},~y_{2}$,$\ldots$,$y_{n}$. 
Thus the vector $(x_1, \ldots, x_t$, $y_2, \ldots, y_t)$, $t=[\frac{k+1}{2}]$,  completely determines $X$.

\medskip  Now, choose a vector $w =(w_1, \ldots, w_k) \in \V$ such that $p(T)^{d-1} w \neq 0$. Choose, $X$ as in \eqnref{hma} such that $w=(y_t, \ldots, y_2, x_1,  \ldots, x_{k-t+1})$ is the $t$-th row of $X$. Then $w, Tw, \ldots$, $T^{k-1}w$ are the rows of $X$. We claim that they must be linearly independent. If not, then assume $w$ is a linear combination of the other rows. This would imply that $p(T)^s w=0$ for $s<d$. This would give a contradiction. So, we can choose $X$ to be non-degenerate.   

Choose $y_i=  x_i^c$, $k \neq 1$,  and $x_1 =x_1^c$ in $X$, and we get the required $c$-hermitian form. 
\end{proof} 

  \begin{cor}\label{c1}
 Let $\V$ be a finite dimensional vector space over a field $\F$ with non-trivial involution $c$. 
Let $T: \V \to \V$ is a non-unipotent, cyclic, self-dual linear transformation. Then there exists a $T$-invariant non-degenerate $c$-hermitian form on $\V$. If $char(\F)\neq 2$, there also exists a $T$-invariant $c$-skew hermitian form. 
\end{cor}
 \begin{proof}
 If $\char(\F) \neq 2$, the existence of  skew-hermitian form is similar noting that there always exists an element $w$ in $\F$ such that $w^c=-w$. 
\end{proof}

\subsection{Invariant form under unipotent linear maps}
Let $\F$ be a field with a non-trivial involution $c$. Then $\F$ is the quadratic extension of a field $$\E=\{x \in \F: \ c(x)=x\}.$$
 Let $\F=\E(w)$. If $\char(\F) \neq 2$, we can take $w^2=a \in \E$ and $w^c=-w$. If $\char(\F)=2$, then we can take $w^2+w \in \E$ and $1+w+w^c=0$. We prove it for the characteristic $2$  case as that is relevant to us.

 Let $\F=\E(w)$. Let $w^{c}= a+bw$, for $a$, $b$ in $\E$.  Clearly, $a$ cannot be zero, hence we can assume $a=1$. Then $(w^{c})^{c}=1+bw^{c}$ implies  $bw^{c}= 1+w$, and hence,  $w^{c}=b^{-1}+b^{-1}w$. Thus $b^{-1}=a$ and $b^{-1}=b$, that implies $1+w+w^{c}=0$. 

\begin{lemma}\label{lem2}
Let $\F$ be a field of characteristic $2$  with a non-trivial involution $c$.  Let $\V$ be an $n$-dimensional vector space of dimension $\geq 2$ over  $\F$.    Let $T: \V \to \V$ be a unipotent linear map with minimal polynomial $(x-1)^n$. Suppose $\V$ is  $T$-indecomposable. Then  $\V$ admits a $T$-invariant  
$c$-hermitian form over $\F$.
\end {lemma}

\begin{proof} 

Let $T$ be an unipotent linear map. Suppose the minimal polynomial of $T$ is $m_T(x)=(x-1)^n$. Without loss of generality we can assume that $T$ is of the form
\begin{equation}\label{t}
T=\begin{pmatrix} 1 & 0& 0& 0 & .... \ 0 & 0\\ 1 & 1 & 0& 0 & .... \ 0& 0 \\ 0 & 1 & 1 & 0 & ....\ 0& 0\\ &\ddots & \ddots  & \ddots    &\ddots& \\ 0 &0 & 0 & \ldots & 1&0\\ 0 & 0 & 0 & \ldots & 1& 1 \end{pmatrix}
\end{equation}
Suppose $T$ preserves a $c$-sesquilinear form $H$. In matrix form, let $H=(a_{ij})$.
Then,  $(T^{c})^{t} H T =H$.
This gives the following relations: For $1 \leq i \leq n-1$,
\begin{equation}\label{1}
 a_{i+1,n}=0=a_{n, i+1},
\end{equation}
\begin{equation}\label{2'}
a_{i,j} + a_{i,j+1}+a_{i+1,j}+a_{i+1,j+1}=a_{i,j},
\end{equation}
\begin{equation}\label{e2}
 \hbox{i.e. } \  a_{i,j+1}+a_{i+1,j}+a_{i+1,j+1}=0. \end{equation}
From the above 2  equations, we have for $1 \leq l \leq n-3$ and $l+2 \leq i \leq n-1$,
\begin{equation}\label{3}
a_{i,n-l}=0=a_{n-l,i}.
\end{equation}
This implies that $H$ is a triangular matrix of the form
\begin{equation}\label{can}
H= \begin{pmatrix} a_{1,1} & a_{1,2} & a_{1,3} & a_{1,4} & .... & a_{1, n-2} & a_{1,n-1} & a_{1,n}\\
a_{1,2} & a_{2,2} & a_{2,3} & a_{2,4} & .... & a_{2,n-2} & a_{2,n-1} & 0\\
a_{1,3} & a_{2,3} & a_{3,3} & a_{3,4} & .... & a_{3,n-2} & 0 & 0\\
\vdots& \vdots&\vdots  &  \vdots& .... &\vdots  & \vdots & \vdots \\
a_{1,n-1} & a_{2,n-2} & 0 & 0 & ....& 0 & 0 & 0 \\
a_{1,n} & 0 & 0 & 0 & ....& 0& 0 & 0 \end{pmatrix},
\end{equation}
where, 
\begin{equation*} \label{e2'} a_{i+1, j} + a_{i,  j+1} + a_{i+1, j+1}=0. 
\end{equation*}
In particular, it follows that $a_{j, n-j+1}=a_{j+1, n-j}$, where $1 \leq j \leq n-1$. Thus $a_{1, n}$ must be chosen non-zero for $H$ to be non-singular.   It also follows from this relation and \eqnref{e2} that $a_{l, n-l}=(-1)^{n-2l} a_{n-l, l}$, and since the characteristic of the base field is $2$, we have $a_{l, n-l}=a_{n-l, l}$. Thus $H=H^t$. For $H$ to be hermitian we shall further have $a_{i, j}=a_{j, i}^c$.

\medskip 
Case 1: Suppose $n$ is even, $n=2m$. Then  by \eqnref{e2}   one can find the following relations as follows:
\begin{eqnarray*} 
a_{2,2m-1}+a_{1,2m}+a_{2,2m}&=& 0\\
a_{3,2m-2}+a_{2,2m-1}+a_{3,2m-1}&=& 0 \\
 \vdots& \vdots& \vdots\\
a_{2,2m-1}^{c}+a_{2m,2}+a_{1,2m}^{c}&=&0. 
\end{eqnarray*}
It follows that  for $1 \leq j \leq n-1$, $a_{j, n-j+1}=a_{j+1, n-j}=a_{2,2m-1}$.
 Now expanding with respect to the final row the determinant of $H$ would be $(a_{2,2m-1})^{2m}$. Thus choosing $a_{2,2m-1}$ to be $1$, it provides a non-singular hermitian form. 

\medskip Case 2: Suppose $n=2m+1$. In this case using \eqnref{e2}, we have in particular 
$$a_{m+1, m+1}+a_{m-1, m+1}+a_{m+1, m-1}=0, \hbox{ i.e. }1+a_{m-1, m+1}+a_{m-1, m+1}^c=0.$$
So, we can choose $a_{k, n-k}=w$ and $a_{k+1, n-k-1}=w^c$, for $1 \leq k \leq m-1$. Next we choose $a_{k, n-k-1}=1$ and $a_{k-1, n-k}=w$ for $2 \leq k \leq m-1$, thus $a_{k, n-k}+a_{k-1, n-k}+a_{k, n-k-1}=1+w+w^c=0$, and so on. By doing this recursively in \eqnref{e2}, we get a non-singular hermitian matrix $H$.  
This completes the proof. 
\end{proof}

\subsection{Proof of \thmref{mainth}}\label{proof}  Suppose that the linear map $T$ admits an invariant non-degenerate hermitian form  $H$.
 Then the necessary condition follows from existing literatures, for example see \cite{wall, s1, s2}. 

\medskip Conversely, let  $\V $  be  a  vector  space  of  dim $  n\ge 2 $ over  the  field
$ \F $ and $ T : \V \longrightarrow \V $ an  invertible  map  such  that  the an
elementary  divisors of $ T$ is either   self-dual or, its dual is also an elementary divisor. For an elementary divisor $g(x)$, let $\V_g$ denote the $T$-indecomposable subspace  isomorphic to $\F[x]/(g(x))$.
From the theory of elementary divisors it follows that $\V$ has the decomposition, see \cite{wall, jac}, 
\begin{equation}\label{x}
\V=\bigoplus_{i=1}^{m_1}\V_{f_i} \oplus  \bigoplus_{j=1}^{m_2} \U_{g_i},
\end{equation} 
where,  for each $i=1,2,...,m_1$,   $f_i(x)$ is either self-dual with no root $1$, or $(x-1)^k$ and,  for each $j=1,2,...,m_2$, $\U_{g_i}=\V_{g_i} \oplus \V_{g_i^{\ast}}$,  $g_i(x)$, $g_i^{\ast}(x)$ are dual to each other,  $g_i(x)\neq g_i^{\ast}(x)$; note that $\dim \V_{g_i}=\dim \V_{g_i^{\ast}}$.  Here $\oplus$ denotes the direct sum. To prove the theorem it is sufficient to induce a $T$-invariant hermitian form on each of the summands. After we have \lemref{lem1} and \lemref{lem2}, the proof goes similarly as in the proof of \cite[Theorem 1.1]{gm} or \cite[Theorem 1.1]{gk}. 

\section{ Proof of \thmref{mainthm}}\label{mpr}
In this section, we shall prove the main theorem. We reformulate the statement of the main theorem as follows and will prove it in the rest of the section.
\begin{theorem} \label{thm1}
Let $\F$ be a field with a non-trivial  involution $c$. Let $T: \V \to \V$ be an invertible linear map. Then $\V$ admits a $T$-invariant non-degenerate $c$-hermitian form if and only if $T$ is $c$-real. 
\end{theorem}

First we shall prove the following lemma. 
\begin{lemma} \label{lem3} Let $\F$ be a field with involution $c$ and $T \in{\rm GL}_k(\F)$, $k \geq 2$. Then $T$ is $c$-real if and only if an elementary divisor of $T$ is either self-dual, or its dual is also an elementary divisor with the same multiplicity.
\end{lemma}

\begin{proof} Let $\V$ be a vector space of dimension $n$ over $\F$ and $T: \V \to \V$ be an invertible linear map. 
Suppose,  $T$ is $c$-real, that is $T$ is conjugate to $(T^c)^{-1}$. By the structure theory of linear maps, they have the same elementary divisors, and hence the same primary decomposition, see \cite{jac}. $T$ is conjugate to $(T^c)^{-1}$ on each of the summands in this decomposition,  and that implies the necessary condition.

 Suppose, an elementary  divisor of $T$ is either  self-dual or, its dual is also an elementary divisor.  Then,  $\V$ has the primary decomposition as in  \eqnref{x}. Hence, it is sufficient to prove the lemma on each of the summands in the decomposition. So, without loss of generality, we assume $\V$ to be one of these summands and prove that $T$ is $c$-real. 

\medskip Case (i): Suppose $\V=\V_f$, for some self-dual monic polynomial $f(x)=p(x)^d$ that has no root $1$.  Then $\chi_T(x)=m_T(x)=p(x)^d=\sum_{i=0}^{ k }d_ix^{i}$ belongs to $ F[x]$, $k=\dim \V$. As $T$ is self-dual then $d_0d_0^c=1$, $d_i d_0^c=d^{c}_{ k -i}, ~ 1\leq i\leq  k -1$. The set 
$$B=\{e_0=v,~ e_1=Tv, \ldots,~ e_{ k-1 } =T ^{ k -1}v\}$$
form a basis for $\V$ and $T^{ k }v=-\sum_{i=0}^{k-1} d_{i}T^{i}v$. The linear transformation  $T$ with respect to the basis $B$ is given by: $T(e_i)=e_{i+1}$, and this is represented by the matrix:  
 $$T=\begin{pmatrix}
0 & 0 & 0 & \ldots & 0 &  -d_0 \\1& 0 & 0 & \ldots & 0  & -d_1\\ 0 &  1& 0& \ldots & 0 &  -d_2 \\ \vdots & & \ddots& \ddots &  & \vdots\\ 0 & 0 & 0 & & 0 & -d_{ k -2}\\0 & 0 & 0 & & 1 & -d_{ k -1} \end{pmatrix}.  $$
It is easy to see that
$$(T^c)^{-1}=\begin{pmatrix}   

-\displaystyle \frac{d_{1}^c}{d_0^c}  & 1 & 0 & \ldots & 0 & 0  \\ - \displaystyle\frac{d_2^c}{d_0^c} & 0 & 1 & \ldots& 0 & 0  \\  \vdots & \ &   \ddots& \ddots & & \vdots   \\  -\displaystyle\frac{d_{k-1}^c}{d_0^c} & 0 & 0 & \ldots & 0&  1 \\
-\displaystyle\frac{1}{d_0^c} & 0 & 0 & \ldots&  0 & 0 \end{pmatrix}.  $$
Now the transformation given by $S(e_i)=e_{k-i}$ conjugates $T$ to $(T^c)^{-1}$. In matrix representation with respect to the basis $B$,  
$$S=\begin{pmatrix} 0 & 0 & \ldots &0 & 0 & 1\\0 & 0  & \ldots &0 & 1  & 0\\ 0 &  0& \ldots & 1 & 0 &  0 \\ \vdots & \vdots & \vdots& \vdots & \vdots & \vdots\\ 1 & 0  & \ldots&0  & 0 & 0 \end{pmatrix}.  $$

Case (ii): Suppose $\V=\U_{g}= \V_{g} \oplus \V_{g^{\ast}}$,  where $g(x)$, $g^{\ast}(x)$ are powers of an irreducible polynomial over $\F$, and are dual to each other.  In this case, with respect to the cyclic basis $B$, if $\t=T|_{\V_g}$, then $T|_{\V_{g^{\ast}}}=(\t^c)^{-1}$ with respect to the dual basis $B^{\ast}$.  Thus, with respect to the basis $\{B, B^{\ast}\}$, $T$ has the representation:
$$T=\begin{pmatrix} \t & O \\ O & (\t^c)^{-1} \end{pmatrix},$$
where each block has order $\frac{k}{2}$. Then 
$$(T^c)^{-1}=\begin{pmatrix}  (\t^c)^{-1} & O \\ O & \t \end{pmatrix}.$$
Now the matrix $STS^{-1}=(T^c)^{-1}$, where 
$$S=\begin{pmatrix}  O_{\frac{k}{2}} & I_{\frac{k} {2} }\\ I_{\frac{k}{2}} & O_{\frac{k}{2}} \end{pmatrix},$$
$I_{\frac{k}{2}}$ denote the identity matrix of order $\frac{k}{2}$. 

Case (iii): Let $\V=\V_p$, where $p(x)=(x - 1)^d$. Then $T^c=T$ and hence, by the theory of Jordan canonical form it follows that $T$ is real, see \cite[Chapter 4]{osho}. 

This completes the proof. 
\end{proof}

\subsection{Proof of \thmref{thm1}}  Suppose $c$ is non-trivial. Then it follows from \thmref{mainthm1} that $\V$ admits a $T$-invariant hermitian form if and only if an elementary divisor of $T$ is either self-dual or its dual is also an elementary divisor. In view of \lemref{lem3}, this is equivalent to the $c$-reality of $T$. This proves the theorem. Hence, \thmref{mainthm} follows. 

\begin{remark} 
The classical notion of real elements in a group is closely related to the notion of strongly real elements. An element $g$ in a group $G$ is \emph{strongly real} if it is a product of two involutions in $G$, see \cite{osho}. We can define strong $c$-reality in a $c$-invariant linear group $G$ as follows: an element $g$ is called \emph{strongly $c$-real} if $g$ is conjugate to $(g^c)^{-1}$ by an involution in $G$, i.e. there is an involution $i \in G$ such that $i g i^{-1}=(g^c)^{-1}$. When $c$ is identity, a strongly $c$-real element is strongly real. The proof of \lemref{lem3} shows that:

\begin{cor} Let $\F$ be a field with involution $c$. 
An element of ${\rm GL}_n(\F)$ is $c$-real if and only if it is strongly $c$-real. 
\end{cor}
\end{remark}

\medskip 

\subsection*{Acknowledgement} 
The authors are thankful to Dipendra Prasad, Anupam Singh and Ian Short for their comments on this work.  The second-named author \hbox{acknowledges} a UGC non-NET {fellowship} from Jadavpur University. It is a pleasure to thank the referee for helpful remarks.

\end{document}